\documentclass[12pt,draft]{amsart}
\usepackage[utf8]{inputenc}

\newcommand{\norm}[1]{\left\|#1\right\|}
\newcommand{\scalar}[2]{\langle{ #1},{#2} \rangle}
\newcommand{\set}[1]{\left\{#1\right\}}

\newcommand{\lr}[1]{\left(#1\right)}

\newcommand{\xp}{x^\dag}
\newcommand{\Bj}[1]{B_{#1}}
\newcommand{\asta}{A^\ast A}

\newcommand{\domain}{\mathcal D}
\newcommand{\range}{\mathcal R}
\newcommand{\R}{\mathbb{R}}

\newtheorem{theo}{Theorem}

\newtheorem{lem}{Lemma}
\newtheorem{prop}{Proposition}
\newtheorem{cor}{Corollary}[section]

\newtheorem{assumption}{Assumption}
\theoremstyle{remark}
\newtheorem{rem}{Remark}
\newtheorem{example}{Example}

\newcommand{\xad}{{x_{\alpha}^\delta}}
\newcommand{\yd}{{y^\delta}}
\newcommand{\rad}{{r_{\alpha}^\delta}}
\newcommand{\ra}{{r_{\alpha}}}
\newcommand{\ga}{{g_{\alpha}}}
\newcommand{\xd}{{x^\dagger}}

\newcommand{\xa}{{x_\alpha}}

\newcommand{\Tik}{T_{\alpha}}

\newcommand{\reg}{J}
\newcommand{\pJ}{\partial \reg}
\renewcommand{\mod}{\Phi}
\newcommand{\modi}{\Psi}
\newcommand{\half}{\frac{1}{2}}

\DeclareMathOperator*{\argmin}{argmin}

\title[Penalty-based smoothness conditions]{Penalty-based smoothness
  conditions in convex variational regularization}
\author{Bernd Hofmann}
\address[Bernd Hofmann]{Faculty of Mathematics, Chemnitz University of Technology,   09107 Chemnitz,  Germany}
\email{bernd.hofmann@mathematik.tu-chemnitz.de}
\author{Stefan Kindermann}
\address[Stefan Kindermann]{Industrial Mathematics Institute, Johannes Kepler University Linz, Alternbergergstraße 69, 4040 Linz}
\email{kindermann@indmath.uni-linz.ac.at}
\author{Peter Math\'e}
\address[Peter Math\'e]{Weierstrass Institute, Mohrenstraße 39, 10117 Berlin, Germany}
\email{peter.mathe@wias-berlin.de}
\date{\today}
\begin{document}
\begin{abstract}
The authors study Tikhonov regularization of linear ill-posed problems with a general convex penalty defined on a Banach space. It is well known that the error analysis requires smoothness assumptions.
Here such assumptions are given in form of inequalities involving only the family of noise-free minimizers along the regularization parameter  and the (unknown) penalty-minimizing solution.
These inequalities control, respectively, the defect of the penalty, or likewise, the defect of the whole Tikhonov functional.
The main results provide error bounds for a Bregman distance, which split into two summands: the first smoothness-dependent term does not depend on the noise level, whereas the second term includes the noise level.
This resembles the situation of standard quadratic Tikhonov regularization Hilbert spaces.
It is shown that variational inequalities, as these were studied recently, imply the validity of the assumptions made here. Several examples highlight the results in specific applications.
 \end{abstract}
\maketitle
\section{Introduction}
As a mathematical model for a linear inverse problem, we consider the ill-posed operator equation
\begin{equation} \label{eq:opeq}
A\,x\,=\,y\,,
\end{equation}
where $A$ is a bounded linear operator from an infinite-dimensional Banach space $X$ to an
infinite-dimensional Hilbert space $H$ such that $\range(A)$, the range of $A$, is a non-closed subset of $H$.
Let $\xd \in X$ denote an exact solution of (\ref{eq:opeq}) with properties to be particularized later.
Unlesss specified otherwise, the norm $\|\cdot\|$ in this study  always refers
to that in $H$. We assume that instead of the exact right-hand side $y \in \range(A)$ only noisy data $\yd \in H$
satisfying
\begin{equation} \label{eq:noise}
\|\yd-y\| \le \delta
\end{equation}
with noise level $\delta \ge 0$ are available. Based on $\yd$ we try to recover $\xd$ in a stable approximate manner
by using variational regularization with general \emph{convex} penalty functionals $\reg$.

Precisely, we are going to analyze convergence conditions for minimizers of the Tikhonov
functional
\begin{equation*}
  T_{\alpha}(x;v) := \frac{1} {2} \norm{A x - v}{}^{2} + \alpha
  \reg(x),\quad x\in X,
\end{equation*}
with regularization parameter $\alpha>0$, for exact right-hand sides $v:=y$ and noisy data $v:=\yd$. In this context, we distinguish regularized solutions
\begin{equation}\label{eq:xa} \xa \in \argmin_{x \in X}  T_{\alpha}(x;y)\end{equation}
and
\begin{equation}\label{eq:xad} \xad \in  \argmin_{x \in X}  T_{\alpha}(x;\yd),\end{equation}
respectively. In case of multiple minimizers we select any family of regularized solutions $\xa$ and $\xad$ for $\alpha>0$. As will be seen from the subsequent results, in particular from the discussion in Remark~\ref{rem:multiple}, the specific choice has no impact on the convergence rate results.


We are interested in estimates of the error between $\xad$ and $\xd$ and in proving corresponding convergence rates. In a Hilbert space $X$, the error
norm is a canonical measure in this context, in particular if the penalty $\reg$ is of norm square type. For Banach spaces $X$ and general convex penalties $\reg$, however, norms are not always appropriate measures, and the study~\cite{BurOsh04} introduced alternatively the Bregman distance
\begin{equation}
\label{eq:bregman}
  \Bj\zeta(z;x) := \reg(x) - \reg(z) - \scalar{ \zeta}{x - z},\quad x\in X, \quad \zeta \in \pJ(z) \subset X^*,
\end{equation}
with some subgradient $\zeta$ from the subdifferential $\pJ(z)$ of $\reg$ at the point $z \in  X$,
as a powerful error measure for regularized solutions of ill-posed problems in Banach spaces, see also, e.g.,~\cite{HoKaPoSc07,Resm05,Scherzer09,ScKaHoKa12}.
We stress the fact that the subgradient $\zeta$ is
taken at the first argument in the Bregman distance, and we recall that the Bregman distance is not symmetric in its arguments. {Therefore, we highlight in~\eqref{eq:bregman} the base point~$z$, by indicating the corresponding subgradient, say~$\zeta$.}

It is a classical result that convergence rates for ill-posed
problems require a regularity condition (abstract smoothness condition)
for $\xd$ as otherwise convergence can be arbitrary slow.

For linear problems in Hilbert space, a classical smoothness condition assumes that~$\xd\in \range\lr{A^{\ast}}$. The corresponding Banach
space assumption instead supposes that there is a
subgradient
$$\xi^{\dag}\in\pJ(\xd) \qquad \mbox{with} \qquad \xi^{\dag} = A^{\ast}w,\;\norm{w}\leq R.$$
For convex $\reg$ the Bregman distance is non-negative, and hence~\eqref{eq:bregman} implies that  for all $x \in X$ the inequality $$
\reg(\xd) - \reg(x) \le \scalar{\xi^{\dag}}{\xd-x}
$$
holds, and then   it is immediate (cf.~\cite[p.~349]{Ki16}) that we have
\begin{equation}
  \label{eq:vi-benchmark}
\reg(\xd) - \reg(x) \leq R \norm{A\xd - Ax} \quad \mbox{for all} \quad  x\in X.
\end{equation}
This  represents a \emph{benchmark variational inequality} as Section~\ref{sec:relations} will reveal.
If otherwise $\xi^{\dag}\not\in\range(A^{\ast})$, then a condition of type (\ref{eq:vi-benchmark}) must fail, but as shown in~\cite[Lemma~3.2]{Fl17}
a variational inequality
\begin{equation}
  \label{eq:vi-Phi}
\reg(\xd) - \reg(x) \leq \mod\lr{\norm{A\xd - Ax}} \quad \mbox{for all} \quad  x\in X
\end{equation}
with a sub-linear \emph{index function}\footnote{Throughout, we call a function $\varphi\colon (0,\infty) \to (0,\infty)$
index function if it is continuous, strictly increasing and obeys the
limit condition $\lim_{t\to +0} \varphi(t) = 0$.}~$\mod$ holds, and
the quotient function $\mod(t)/t$ is strictly decreasing.
Under a more restrictive assumption on~$\mod$ (concavity instead of sub-linearity)
and for a more general setting this condition was introduced as formula (2.11) in \cite{Ki16}, and it was proven that (\ref{eq:vi-Phi}) yields the convergence rates
\[ \Bj{\xi_\alpha^\delta}(\xad;\xd) = \mathcal{O}(\mod(\delta)) \quad
\mbox{as} \quad \delta \to 0\,. \]
In the past years, convergence rates under variational source conditions (cf., e.g.,~\cite{Fl12,Gr10,HoMa12})  were expressed in terms of  the Bregman distance~$B_{\xi^\dagger}(\xd;\xad)$, and hence using the base point~$\xd$.  In this context it is not clear whether a subgradient~$\xi^\dag\in \pJ(\xd)$ exists, for instance if~$\xd$ is not in the interior of~$\mathrm{dom}(\reg):=\{x \in X:\,J(x)<\infty\}$. Taking as base point the minimizer~$\xad$, this cannot happen and the set~$\pJ(\xad)$ is always non-empty (cf., e.g.,~\cite[Lemma~2.2]{Fl17}). This may be seen as an advantage of the present approach, following the original study~ \cite{Ki16}.

Without
further notice we follow the convention from that study: if the
subdifferential $\partial J(\xad)$ is multi-valued, then we take for
$\Bj{\xi_\alpha^\delta}(\xad;\xd)$ a subgradient $\xi_\alpha^\delta$
that satisfies the optimality condition
\begin{equation} \label{eq:left0}
A^*(A\xad-\yd)+\alpha \xi_\alpha^\delta=0.
\end{equation}

A remarkable feature of the error bounds under smoothness
assumptions~(\ref{eq:vi-Phi}) is the splitting of the error, see
also in a more general setting~\cite[Thm.~3.1]{Ki16}, as
\begin{equation}
  \label{eq:err-split}
B_{\xi_{\alpha}^{\delta}}(\xad;\xd) \leq \frac{\delta^{2}}{2\alpha} +
\modi(\alpha) \quad \mbox{for all} \quad  \alpha>0,
\end{equation}
where the function~$\modi$ is related to~$\mod$, and typically it will
also be an index function.

\medskip

In this study, see Section~\ref{sec:main}, we analyze the condition
\begin{equation}
  \label{eq:weaker}
  \reg(\xd) - \reg(\xa) \leq \modi(\alpha) \quad \mbox{for all} \quad  \alpha>0
\end{equation}
(cf.~Assumption~\ref{ass:conv}, and its alternative,~Assumption~\ref{ass:2prime}). Under
these conditions a similar error splitting as in~(\ref{eq:err-split})
is shown as the main result. Notice, that these bounds are required to
hold only for the minimizers~$\xa$ of the noise-free Tikhonov
functional~$\Tik(x;A\xd)$. This resembles the situation for linear
ill-posed problems in Hilbert spaces, where the error is decomposed
into the noise propagation term, usually of the
form~$\delta/\sqrt\alpha$,  and some noise-free
term depending on the solution smoothness, say~$\varphi(\alpha)$, which is called profile function in \cite{HofMat07}. We refer to a detailed
discussion in Section~\ref{sec:main}.
The error bounds will be  complemented by some
discussion on the equivalence of  Assumptions~\ref{ass:conv} and~\ref{ass:2prime}. Also, a discussion on necessary conditions
for an index function $\modi$ to serve as an inequality~(\ref{eq:weaker}) is
given. We mention that the existence of an index function $\modi$ satisfying (\ref{eq:weaker}) is an immediate consequence of \cite[Thm.~3.2]{Fl18} (see also \cite[Remark~2.6]{Fl17}) in combination with the results of Section~\ref{sec:relations}. Precisely, we highlight that the variational
inequality (\ref{eq:vi-Phi}) implies the validity of~(\ref{eq:weaker})
for some specific index functions $\modi$ related to~$\mod$ by some
convex analysis arguments. Then, in Section~\ref{sec:examples}
we present specific applications of this  approach.

In an appendix we
give detailed proofs of the main results (Appendix~\ref{sec:proofs})
and some auxiliary discussion concerning convex index functions
(Appendix~\ref{sec:convex-analysis}).

\section{Assumptions and main results}
\label{sec:main}

In the subsequent analysis, \emph{convex} index functions will be of particular interest, i.e.,~index functions $\varphi$ which obey
$$
\varphi\left(\frac{s+t}{2}\right) \leq \frac 1 2 \lr{\varphi(s) + \varphi(t)},\quad s,t\geq 0.
$$
The inverse of a convex index function is a \emph{concave} index function, and hence the above inequality is reversed. We mention that concave index functions are \emph{sub-linear}, which means that these functions have the property that the quotients~$\varphi(\lambda)/\lambda$ are non-increasing.
Additional considerations concerning convex index functions are
collected in Appendix~\ref{sec:convex-analysis}.

The proofs of the results in this section are technical, and hence
they are postponed to Appendix~\ref{sec:proofs}.

\subsection{Assumptions}
\label{sec:assumptions}

Throughout this study we impose, e.g.,~along the lines
of~\cite{ScKaHoKa12}, the following standard assumptions on
the penalty.
\begin{assumption}[Penalty]\label{ass:penalty}
  The function~$\reg: X \to [0,\infty]$ is a proper, convex functional defined on an Banach space $X$,
which is lower semi-continuous with respect to weak (or weak$^*$) sequential
convergence. Additionally, we assume that
$\reg$ is a stabilizing (weakly coercive) penalty functional, i.e., the sublevel sets
$\mathcal{M}_c:=\{x \in X:\,\reg(x) \le c\}$ of $\reg$ are for all $c \ge 0$ weakly (or weak$^*$) sequentially compact. Moreover, we assume that at least one solution $\xd$ of (\ref{eq:opeq}) with finite penalty value $J(\xd)<\infty$ exists.
\end{assumption}
Consequently, for all $\alpha>0$ and $v \in H$, the sublevel sets of $T_{\alpha}(.,v)$ are weakly (or weak$^*$) sequentially compact. This ensures the existence and stability of
regularized solutions $\xa$ and $\xad$ which are the corresponding minimizers for $v=y$ and $v=\yd$, respectively. In the sequel, we use the symbol $\xd$ only for the always existing $\reg$-minimizing solutions of (\ref{eq:opeq}), i.e.~$J(\xd)=\min \limits_{x \in X: Ax=y} J(x)$.

The fundamental regularity condition is given as follows. To this end,
let $\xa$ be defined as in \eqref{eq:xa}. This assumption controls the
deviation of the penalty at the minimizers from the one at the $\reg$-minimizing solution~$\xd$.
\begin{assumption}[Defect for penalty]\label{ass:conv}
  There is an index function~$\modi$ such that
\begin{equation} \label{eq:new}
\reg(\xp) - \reg(\xa) \leq \modi(\alpha) \quad \mbox{for all} \quad  \alpha>0.
\end{equation}
\end{assumption}

\medskip

It is not difficult to conclude from the minimizing property
of~$\xa$,
\begin{equation} \label{eq:minprop}
\frac{1} {2} \norm{A \xa - A\xd}{}^{2} + \alpha J(\xa) \le  \alpha J(\xd) \,,
 \end{equation}
that the left hand side of \eqref{eq:new}
is nonnegative and hence that
$$\lim_{\alpha \to 0}\reg(\xa) =\reg(\xp) \quad \mbox{and} \quad \frac{1}{2 \alpha} \norm{A \xa - y}{}^{2} \le \reg(\xp) - \reg(\xa), $$
such that  Assumption~\ref{ass:conv} also yields the estimate
\begin{equation}
  \label{eq:residual-bound}
 \frac{1}{2\alpha}\norm{A \xa - y}{}^{2} \leq \modi(\alpha) \quad \mbox{for all} \quad  \alpha>0.
\end{equation}
Instead of controlling the defect for the penalty~$\reg$ one might
control the defect for the overall Tikhonov functional as follows.

\renewcommand{\theassumption}{2$^{\,\prime}$}
\begin{assumption}[Defect for Tikhonov functional]\label{ass:2prime}
  There is an index function~$\modi$ such that
\begin{equation}\label{eq:Modified}
\frac{1}{\alpha}\left( \Tik(\xd;A\xd) - \Tik(\xa;A \xd) \right) \leq \modi(\alpha) \quad \mbox{for all} \quad  \alpha>0.
\end{equation}
\end{assumption}

\medskip

By explicitly writing the left hand side in~(\ref{eq:Modified}) we see
that
$$
\frac{1}{\alpha}\left( \Tik(\xd;A\xd) - \Tik(\xa;A \xd) \right) =
\reg(\xd) - \reg(\xa) - \frac{1}{2\alpha}\norm{A \xa - y}{}^{2},
$$
and hence Assumption~\ref{ass:conv} is stronger than
Assumption~\ref{ass:2prime}, as stated above.
One advantage of Assumption~\ref{ass:2prime} is that it is invariant
with respect to the choice of the minimizers~$\xa$. This is not clear for Assumption~\ref{ass:conv}.
As a remarkable fact we state that both assumptions are
basically equivalent.
\begin{prop}\label{thm:equivalence}
  Assumption~\ref{ass:2prime} yields that
  \begin{equation}
    \label{eq:equivalence3}
    \reg(\xp) - \reg(\xa) \leq 2 \modi(\alpha) \quad \mbox{for all} \quad  \alpha>0.
  \end{equation}
Hence Assumption~\ref{ass:conv} is fulfilled with~$\modi$ replaced by $2\modi$.
\end{prop}

\begin{rem}\label{rem:multiple}
The above result has an important impact, and we return to the choice
of the minimizers~$\xa,\xad$ from \eqref{eq:xa} and~\eqref{eq:xad},
respectively. As mentioned before, the functional on the left-hand side of \eqref{eq:Modified} is independent of the choice of the minimizers~$\xa$, due to the uniqueness of the value of the Tikhonov functional at the minimizers
(cf., e.g.,~\cite[Sec.~3.2]{ItoJin15}). Thus, if Assumption~\ref{ass:2prime} is fulfilled for one selection~$\xa,$ $\alpha>0,$ then this holds true for arbitrary selections. Since Assumption~\ref{ass:2prime} implies Assumption~\ref{ass:conv} (at the expense of a factor~2) the latter will be fulfilled for any selection.
Conversely, if Assumption~\ref{ass:conv} holds for some selection~$\xa, \ \alpha >0$, then this yields the validity of Assumption~\ref{ass:2prime}, but then extends to any other choice of minimizers. Again, by the above proposition this implies that any other choice of minimizers will obey Assumption~\ref{ass:conv}, by losing a factor $2$ at most.
\end{rem}

We finally discuss which index functions may serve as upper bounds in
either of the assumptions~\ref{ass:conv} or~\ref{ass:2prime},
respectively.
We formulate this as follows.

\begin{prop}  \label{prop:dich}
Suppose Assumption~\ref{ass:conv} holds with index function $\modi$. Then
the following is true:\\
 {\bf Either} $\;J(\xd) = \min \limits_{x\in X} J(x)$,\\
 and then $J(\xa) =J(\xd)$ for each $\alpha>0$, and any index function~$\modi$ is a valid
 bound in \eqref{eq:new},\\
 {\bf or} $\;J(\xd) > \min \limits_{x\in X}  J(x)$,\\
 and then $\modi$ increases near zero at most linearly.
\end{prop}
We shall call the first case \emph{singular}.  In this case, where $J(\xd) = \min _{x\in X} J(x)$, the choice of the regularization parameter loses importance, which is also the case if the phenomenon of exact penalization occurs
(see \cite{BurOsh04} and more recently in~\cite{AnzHofMat14}).

\subsection{Main results}
\label{sec:main-results}

We turn to stating the main results, which highlight the impact of Assumption~\ref{ass:conv} and Assumption~\ref{ass:2prime} on
the overall error, measured  by the Bregman distance.
\begin{theo}\label{th:main}
  Under Assumption~\ref{ass:conv} we have that
$$
\Bj{\xi_\alpha^\delta}(\xad;\xp) \leq \frac{\delta^{2}}{2\alpha} +
\modi(\alpha) \quad \mbox{for all} \quad  \alpha>0.
$$
\end{theo}

The proof of Theorem~\ref{th:main} is a simple consequence of the
following result, which may be also of its own interest.
\renewcommand{\thetheo}{1$^{\,\prime}$}
\begin{theo}\label{thm:mod}
Suppose that Assumption~\ref{ass:2prime}{} is satisfied with an
index function $\modi$.
Then  an error estimate of the type
\begin{equation}\label{eq:upperbound2}
\Bj{\xi_\alpha^\delta}(\xad;\xd) \leq \frac{\delta^2}{2\alpha} +  \modi(\alpha)\quad \mbox{for all} \quad  \alpha>0  \end{equation}
holds.
\end{theo}
Since, as mentioned above, Assumption~\ref{ass:conv} is stronger than Assumption~\ref{ass:2prime} it is enough to prove Theorem~\ref{thm:mod}.

\subsection{Discussion}
\label{sec:discussion}

Resulting from Theorems~\ref{th:main} and Theorem~\ref{thm:mod}, the best possible bound for the Bregman distance between the regularized solutions and $\xd$ as a function of $\delta>0$ takes place in both cases if $\alpha=\alpha_*>0$ is chosen such that
the right-hand side $\frac{\delta^2}{2\alpha} +  \modi(\alpha)$ is minimized, i.e.,
\begin{equation}\label{eq:bestrate}
B_{\xi_{\alpha_*}^\delta}\lr{x_{\alpha_*}^\delta;\xd} \leq  \inf_{\alpha >0}  \set{\frac{\delta^2}{2\alpha} +  \modi(\alpha)},
\end{equation}
which determines, from this perspective, the best possible convergence rate of $B_{\xi_{\alpha_*}^\delta}\lr{x_{\alpha_*}^\delta;\xd}$ to zero as $\delta \to 0$. Consequently,
this convergence rate is the higher, the faster the decay rate of $\modi(\alpha) \to 0$ as $\alpha \to 0$ is. As expressed in the non-singular case of Proposition~\ref{prop:dich}, the function~$\modi$ cannot increase from zero super-linearly, and the limiting case is obtained for~$\modi(\alpha) \sim \alpha,\ \alpha\to 0$.
From this perspective, the maximally described rate is  $B_{\xi_{\alpha_*}^\delta}\lr{x_{\alpha_*}^\delta;\xd}  \sim \delta$ as $\delta \to 0$, which is obtained whenever for example the regularization parameter is chosen as $\alpha_*=\alpha(\delta)\sim \delta$. For linear ill-posed equations in Hilbert spaces and using the standard penalty $J(x)=\|x\|_X^2$ (see Section~\ref{subsec:quadratic}),   this results in the error rate
$B_{\xi_{\alpha_*}^\delta}\lr{x_{\alpha_*}^\delta;\xd}=\norm{x_{\alpha_*}^\delta- \xd}_X^2 = \mathcal O(\delta)$.
However, resulting from Theorems~\ref{th:main} and Theorem~\ref{thm:mod} the overall best possible convergence rate $\norm{x_{\alpha_*}^\delta- \xd}_X^2=\mathcal{O}(\delta^{4/3})$  attainable for Tikhonov regularization cannot be obtained, and indeed our analysis is confined to the low rate case expressed be the range-type source condition $\xd\in\range(A^\ast)$.
This is also the case for all other approaches which are based on the minimizing property $T_\alpha(\xad;\yd) \le T_\alpha(\xd;\yd)$ only, including approaches using variational source conditions (see Section~\ref{sec:relations}
below). For alternative techniques leading to enhanced convergence rates we refer to \cite{NHHKT10,
Resm05,ResSch06}, \cite[Sect.~4.2.4]{ScKaHoKa12} and references therein.

\medskip

Now we return to the error estimate (\ref{eq:upperbound2}) for general convex penalties~$J$. Since the upper bound with respect to $\alpha>0$ is decomposed into a sum of a continuous decreasing function~$\delta^2/(2\alpha)$ and an increasing continuous function
$\modi(\alpha)$, the minimizer always exists. Given~$\modi$, let us assign the companion~$\Theta(\alpha):= \sqrt{\alpha \modi(\alpha)},\ \alpha>0$. If we then let~$\alpha_\ast$ be obtained from calibrating both summands as
\begin{equation}
    \alpha_\ast = \alpha_\ast(\delta) := \lr{\Theta^2}^{-1}\lr{\frac{\delta^2}{2}} = \Theta^{-1}\lr{\frac{\delta}{\sqrt 2}},
\end{equation}
then we find that
\begin{equation}
\Bj{\xi_{\alpha_\ast}^\delta}(x_{\alpha_\ast}^\delta;\xd) \leq 2 \modi\lr{\Theta^{-1}\lr{\frac{\delta}{\sqrt 2}}},
\end{equation}
and the optimality of this bound will be discussed in the examples presented below in Section~\ref{sec:examples}.
\medskip

It is interesting to separately discuss the singular case, i.e., when
$\;J(\xd) = \min \limits_{x\in X} J(x)$. We claim that then~$\Bj{\xi_\alpha^\delta}(\xad;\xd) =0$ when the subdifferential~$\xi_\alpha^\delta=\pJ(\xad)$ obeys the optimality condition (\ref{eq:left0}),
i.e., we have $\alpha\xi_\alpha^\delta = A^\ast\lr{\yd - A\xad}$.
If we now look at the minimizing property of~$\xad$ then we see that
$$
\frac 1 2 \norm{A\xad - \yd}^2 + \alpha \reg(\xad) \leq \alpha \reg(\xd),
$$
which, in the singular case,  requires to have that~$\norm{A\xad - \yd}=0$, and hence that~${\xi_\alpha^\delta} = 0$.
This yields  for the Bregman distance that
\begin{align*}
   B_{\xi_{\alpha}^\delta}\lr{x_{\alpha}^\delta;\xd}
   &=\reg(\xd) - \reg(\xad) + \scalar{\xi_\alpha^\delta}{\xad - \xd} \leq 0,
\end{align*}
such that the Bregman distance equals zero in the singular case.
\medskip

We already emphasized that the upper estimate of the error measure $\Bj{\xi_\alpha^\delta}(\xad;\xd)$ in (\ref{eq:upperbound2}) consists of two terms, the first $\delta$-dependent noise propagation, and the second $\delta$-independent term which expresses the smoothness of the solution $\xd$ with respect to the forward operator $A$. In the study~\cite{HofMat07} such a
decomposition was comprehensively analyzed for general linear regularization methods applied to~(\ref{eq:opeq}) in a Hilbert space setting, i.e.,  for linear mappings $\yd \mapsto \xad$, and for the norm
as an error measure the $\delta$-independent term was called \emph{profile function} there, because this term completely determines the error profile. For the current setting,
the index function $\modi$ plays a similar role, although the mapping $\yd \mapsto \xad$ is \emph{nonlinear} for general convex penalties $\reg$ different from norm squares in Hilbert space $X$.
This shows the substantial meaning of the right-hand function $\mod$ in the inequality (\ref{eq:new}) of Assumption~\ref{ass:conv}.

\section{Relation to variational inequalities}
\label{sec:relations}
In this section we shall prove that a
variational inequality of type (\ref{eq:vi-Phi}) implies the
validity of Assumption~\ref{ass:2prime} and a fortiori Assumption~\ref{ass:conv}.
More precisely, we consider the situation that there is an  index function~$\mod$ such that
\begin{equation}\label{eq:kindermann}
  \reg(\xd) - \reg(x) \leq \mod(\|A x - A \xd \|) \quad \mbox{for all} \quad  x\in X.
\end{equation}
First, similarly to Proposition~\ref{prop:dich} we highlight that the choice of functions~$\mod$ in~\eqref{eq:kindermann} is not arbitrary.

\begin{prop} \label{prop:nosuper}
Suppose that a variational inequality~\eqref{eq:kindermann} holds with an index function~$\mod$.
The following is true:\\
 {\bf Either} $\;J(\xd) = \min \limits_{x\in X} J(x)$,\\
 and then any index function~$\mod$ is a valid
 bound in \eqref{eq:kindermann},\\
 {\bf or} $\;J(\xd) > \min \limits_{x\in X}  J(x)$,\\
 and then $\mod$ increases near zero at most linearly.
\end{prop}
\begin{proof}
First, if~$\;J(\xd) = \min \limits_{x\in X} J(x)$ then the left hand side in~\eqref{eq:kindermann} is non-positive, and hence any non-negative upper bound is valid. Otherwise, suppose that~$\mod(t)/t$ decreases to zero as~$t\to 0$.
The inequality (\ref{eq:kindermann}) taken at the point $\xd+t(x-\xd),\;0<t<1,$ attains the form
$$J(\xd)-J((1-t)\xd+tx) \le \mod(t\|Ax-A\xd\|), $$
where we can estimate from below the left-hand side as
$$ J(\xd)-J((1-t)\xd+tx) \ge J(\xd)-(1-t)J(\xd)-tJ(x)=t(J(\xd)-J(x)), $$
because $J$ is a convex functional. From this we directly derive
$$ J(\xd)-J(x) \le \frac{\mod(t\|Ax-A\xd\|)}{t}= \|Ax-A\xd\|\frac{\mod(t\|Ax-A\xd\|)}{t\|Ax-A\xd\|},$$
where under the assumption of the lemma the right-hand side tends to zero as $t \to 0$. Consequently, we have $J(\xd) \le J(x)$ for all $x \in X$. This completes the proof.
\end{proof}

The main result in this section reads as follows:

\begin{prop} \label{pro:peter}
 Suppose that a variational inequality~(\ref{eq:kindermann}) holds for some index function~$\mod$. Let us consider the related index function  $\tilde{\mod}(t):= \mod(\sqrt{t}),\;t>0$.
 Then the following assertions hold true.
 \begin{enumerate}
     \item The condition~(\ref{eq:Modified}) is valid with
  a function
\begin{equation} \label{eq:infpos}
\modi(\alpha) =   \sup_{t>0} \left[ \mod(t) - \frac{t^2}{2 \alpha}\right],
\end{equation}
which is increasing for all $\alpha>0$ but that may take values~$+\infty$.
\item If the function~$\tilde\mod$ is concave then the function~$\modi$ from (\ref{eq:infpos}) has the representation
\begin{equation*}
    \modi(\alpha):= \frac{\tilde{\mod}^{-\ast}(2\alpha)}{2\alpha},\quad \alpha >0,
\end{equation*}
where $\tilde{\mod}^{-\ast}$ is the Fenchel conjugate to the convex index function  $\tilde{\mod}^{-1}$ (cf.~Appendix~\ref{sec:convex-analysis}).
\item\label{it:3}
Finally, if moreover the quotient function $s^2/\mod(s)$, is an index function and hence strictly increasing for all $0<s<\infty$, then $\modi$ also constitutes an index function. Theorem~\ref{thm:mod}
yields the error estimate (\ref{eq:upperbound2}). \end{enumerate}
\end{prop}
\begin{proof}
For the first assertion we find that
\begin{align*}
&\frac{1}{\alpha} \left(\Tik(\xd;y) - \Tik(\xa;y) \right)
= \reg(\xd) - \reg(\xa)  -  \frac{1}{2\alpha} \|A \xa  - A \xd \|^2  \\
& \leq \mod(\|A  \xa   - A \xd \| )  -
  \frac{1}{2\alpha} \|A \xa  - A \xd \|^2.
\end{align*}
Setting $t:= \norm{A  \xad   - A \xd}$
yields the function~$\modi$ as stated.

Now suppose that the function~$\tilde \mod$ is a concave index function. Then its inverse is a convex index function, and by the definition of the Fenchel conjugate, see~\eqref{eq:Fconjugate}, we  find
\begin{align*}
&\sup_{t>0} \left[\mod(t) - \frac{t^2}{2 \alpha}\right] =
    \sup_{t>0} \left[\tilde{\mod}(t^2) - \frac{t^2}{2\alpha}\right] \\
    & =
       \frac{1}{2 \alpha} \sup_{s>0} \left[ 2 \alpha s - \tilde{\mod}^{-1}(s) \right] =
  \frac{\tilde{\mod}^{-\ast}(2\alpha)}{2\alpha},
  \end{align*}
  which proves the second assertion. It remains to establish that this function is an index function with property as stated. To this end we aim at applying Corollary~\ref{cor:appendix} with~$f(t):= \tilde\mod^{-1}(t),\ t>0$. We observe, after substituting~$t:= \tilde\mod(s^2)$, that
  $$
  \frac{\tilde\mod^{-1}(t)}{t} = \frac{s^2}{\mod\lr{s}},\quad s>0,
  $$
  which was supposed to be strictly increasing from 0 to $\infty$. Thus Corollary~\ref{cor:appendix} applies, and the proof is complete.
\end{proof}

Under the conditions of item (2) of Proposition~\ref{pro:peter} we can immediately derive a convergence rate for the Bregman distance as error measure.

\begin{prop} \label{pro:stefan}
If the function $\mod$ in (\ref{eq:kindermann}) is such that
$\tilde{\mod}(t):= \mod(\sqrt{t})$ is
a concave index function, and
with an appropriately selected  $\alpha$,
the following convergence rate  holds
\begin{equation} \label{eq:rateStefan}
\Bj{\xi_\alpha^\delta}(\xad;\xd) = \mathcal{O}(\mod(\delta)) \quad
\mbox{as} \quad \delta \to 0.
\end{equation}
\end{prop}
\begin{proof}
In the Fenchel-Young inequality~\eqref{eq:FYI}, used for~$f:= \tilde{\mod}^{-1}$,
assigning~$u:= \tilde{\mod}(\delta^2)$ and~$v:= 2\alpha$ we obtain
\[ \mod(\delta) = \tilde{\mod}(\delta^2) \leq \frac{\delta^2}{2 \alpha} +
\frac{\tilde{\mod}^{-*}(2 \alpha)}{2\alpha} \]
Taking $2 \alpha \in \partial \tilde{\mod}^{-1}(\delta^2)$, which
exists by continuity of $\tilde{\mod}^{-1}$, yields
equality in the Fenchel-Young and in the above inequality, thus,
with such a choice and by \eqref{eq:upperbound2}
\[  \Bj{\xi_\alpha^\delta}(\xad;\xd) \leq
    \frac{\delta^2}{2\alpha} + \frac{\tilde{\mod}^{-*}(2 \alpha)}{2\alpha} = \mod(\delta).\]
\end{proof}

We highlight the previous findings in case that the function~$\mod$ in~\eqref{eq:kindermann} is a monomial.
\begin{example}\label{xmpl:monomial}
  Let us prototypically consider the case that the function~$\mod$ is
  of power type, i.e.,\ $\mod(t):= t^{\mu},\ t>0$ for some
  $0 < \mu <\infty$. Then the function~$\tilde\mod$
  is~$\tilde\mod(t)=t^{\mu/2}$.  This function is concave
  whenever~$0< \mu \leq 2$. In that range also the
  quotients~$s^2/\mod(s),\ s>0$ are strictly increasing.  For~$\mu>2$
  the function~$\modi$ is infinite for all~$\alpha >0$ and for $\mu=2$ it is a positive constant. For $0<\mu<2$, however, $\modi$ is an index function.

  Namely, the inverse of~$\tilde\mod$
  equals~$\tilde\mod^{-1}(t) = t^{2/\mu},\ t > 0$. By using the simple identity that~$(cf)^\ast(t) = c f^\ast(t/c),\ t>0$, for a convex function~$f$ and $c>0$ we see that the Fenchel conjugate
  function is for all $0<\mu<2$
  $$
  \tilde\mod^{-\ast}(t) = \frac{2-\mu}{\mu}\lr{\mu t/2}^{2/(2-\mu)},\ t > 0.
  $$
  Then the quotient
$$
\frac{\tilde{\mod}^{-\ast}(2\alpha)}{2\alpha} = \frac{2 - \mu}{2}\lr{\mu\alpha}^{\frac{\mu}{2-\mu}},\quad \alpha > 0,
$$
is a strictly increasing index function as predicted by the
proposition.   This function is sub-linear for~$\mu/(2-\mu)\leq 1$, i.e.,\
for~$0< \mu \leq 1$, and hence may serve
as a bound in Assumption~\ref{ass:conv}, including the benchmark
case~$\mod(t) = ct,\ t>0$, in which case the corresponding
function~$\modi$ is also linear.
\end{example}

\begin{rem} \label{rem:Flemming}
We know from Proposition~\ref{prop:nosuper} that in the non-singular case the function~$\mod$ is at most linear, i.e.,\ the function~$\mod(s)/s$ is bounded away from zero. In particular this holds for concave index functions. In this particular case the function~$s/\mod(s)$ is non-decreasing, and hence the function  $s(s/\mod(s))=s^2/\mod(s)$ is an index function. Thus item~\eqref{it:3} of Proposition~\ref{pro:peter} applies and yields that Assumption~\ref{ass:2prime} holds. Hence Theorem~\ref{thm:mod} applies and gives a convergence rate.
\smallskip

Note that \eqref{eq:kindermann} with the function~$\mod(t)=Rt$, has benchmark character.
Indeed, if~(\ref{eq:kindermann}) holds with an index function~$\mod$ obeying $0<R=\lim_{\alpha \to 0} \mod(\alpha)/\alpha\leq R< \infty$, then this implies the variational inequality
$$\reg(\xd) - \reg(x) \leq R \norm{A\xd - Ax} \quad \mbox{for all} \quad  x\in X.$$
This was shown to hold if~$\range(A^*) \cap \partial J(\xd)\not =\emptyset$, cf. Eq~\eqref{eq:vi-benchmark}.
If such linear bound fails then by the method of \emph{approximate variational source conditions} (cf.~\cite{FleHof10} and more comprehensively \cite{Fl12}) one can consider the strictly positive and decreasing \emph{distance function}
$$d(R) := \sup_{x\in X}\set{\reg(\xd) - \reg(x) - R  \norm{A\xd - Ax}},\quad R>0.$$
 We find that~$\lim_{R \to \infty} d(R)=0$, and  the decay rate to zero as $R \to \infty$ measures the degree of violation of the benchmark variational inequality~(\ref{eq:vi-benchmark}). Together with \cite[Lemma~3.2]{Fl17} it was proven that then a variational inequality of type (\ref{eq:kindermann}) holds, such that
\begin{equation} \label{eq:distphi}
\mod(\alpha)= 2d(\Theta^{-1}(\alpha)), \quad \mbox{where} \quad \Theta(R):=d(R)/R.
\end{equation}
It should be noted that this function $\mod$ is a sub-linear index function such that the quotient function $\mod(\alpha)/\alpha$ is non-increasing for all \linebreak $\alpha>0$. Hence, the convergence rate (\ref{eq:rateStefan}) also applies
for the function~$\mod$ from~(\ref{eq:distphi}).
\end{rem}

\section{Examples}
\label{sec:examples}
Here we shall highlight the applicability of the main results in special situations. We start with the standard penalty in a Hilbert space context and then analyze other penalties as these are used in specific applications.
\subsection{Quadratic Tikhonov regularization in Hilbert spaces} \label{subsec:quadratic}
  Suppose we are in the classical context of Tikhonov regularization in Hilbert spaces~$X$ and $Y$, where the penalty is given as~$\reg(x) := \frac 1 2
  \norm{x}{}^{2},\ x\in X$. 
%
In this case, which has been comprehensively discussed in the literature (cf., e.g.,~\cite[Chap.~5]{EHN96} and \cite{AlbElb16,AndElb15,FHM11}), we can explicitly calculate the terms under
consideration.

First, let~$\ga(\lambda) := 1/(\alpha + \lambda)$ be the filter from
Tikhonov regularization and its companion~$\ra(\lambda) =
\alpha/(\alpha + \lambda)$. With these short-hands, we see that~$\xa -
\xp = \ra(\asta) \xp$, and also~$A(\xa - \xp) = A\ra(\asta)\xp$.
This yields
\begin{equation}\label{eq:axa-axp}
  \frac{\norm{A\xa - A\xp}{}^{2} }{2\alpha} = \frac{\norm{A\ra(\asta)\xp}{}^{2}}{2\alpha}
 = \frac{ \norm{\ra(\asta)\lr{\asta}^{1/2}\xp}{}^{2}}{2\alpha}.
\end{equation}
We also see that
\begin{align*}
  T_{\alpha}(\xa;A\xp) &=  \frac{1} 2 \lr{
                         \norm{\ra(\asta)\lr{\asta}^{1/2}\xp}^{2}
                         + \alpha \norm{\xa}^{2}}\\
& = \frac 1 2 \int \left[ \frac{\alpha^{2}\lambda}{(\alpha +\lambda)^{2}}  + \frac
  {\alpha\lambda^{2}}{(\alpha + \lambda)^{2}}\right]dE_\lambda
  \|\xd\|^2\\
& = \frac 1 2 \int \frac{\alpha\lambda}{(\alpha + \lambda)} dE_\lambda  \|\xd\|^2,
\end{align*}
which in turn yields
\begin{equation}
  \label{eq:kindermann-functional}
  \begin{split}
\frac{1}{\alpha} \left(\Tik(\xd;y) - \Tik(\xa;y) \right) & = \frac 1 2
                                                           \int
\frac{ \alpha  }{(\lambda +\alpha)}   dE_\lambda \|\xd\|^2\\
& = \frac 1 2 \norm{\ra^{1/2}(\asta)\xp}^{2}.
  \end{split}
\end{equation}
Finally, we bound
\begin{equation}\label{eq:diff-J}
  \begin{split}
  \reg(\xp) - \reg(\xa) & = \frac 1  2 \lr{\norm{\xp}{}^{2} - \norm{\xa}{}^{2}}
= \frac 1 2 \scalar{\xp - \xa}{\xp + \xa}\\
& = \frac 1 2 \scalar{\ra(\asta) \xp}{\lr{I +  \lr{\alpha +
  \asta}^{-1}\asta}\xp}\\
& \leq \norm{\ra^{1/2}(\asta)\xp}^{2}.
  \end{split}
\end{equation}
We observe that the right-hand sides in~(\ref{eq:kindermann-functional})
and~(\ref{eq:diff-J}) differ by a factor $\half$,
as predicted in Proposition~\ref{thm:equivalence}.

In the classical setup of Tikhonov regularization, a regularity condition
is usually imposed by a source-condition.
Thus, let us now assume that the element~$\xp$ obeys a source-wise
representation~
\begin{equation}\label{eq:sc} \xp = \varphi(\asta)v,\qquad  \norm{v}\leq 1,\end{equation}
for an  for an index function~$\varphi$.
Then the
estimate~(\ref{eq:diff-J}) reduces to bounding
\begin{equation}
  \label{eq:diff-J-phi-bound}
  \begin{split}
\norm{\ra^{1/2}(\asta)\xp}^{2} & \leq
\norm{\ra^{1/2}(\asta)\varphi(\asta)}^{2}\\
&\leq  \norm{\ra(\asta)\varphi^{2}(\asta)},
  \end{split}
\end{equation}
where we used the estimate~$\norm{H^{1/2}}{}\leq \norm{H}^{1/2}$ for a
self-adjoint non-negative operator~$H$.

Then, if the function~$\varphi^{2}$ is sub-linear, we find that
$$
\norm{\ra(\asta)\varphi^{2}(\asta)}\leq \varphi^{2}(\alpha).
$$
Hence, in the notation of \cite{MaPe03},
$\varphi^{2}$ is a
  qualification for Tikhonov regularization, and Assumption~\ref{ass:conv}
  holds true with the index function $\modi(\alpha) = \varphi^{2}(\alpha)$.
In particular, the rate~\eqref{eq:bestrate}, which is obtained by equilibrating both summands by letting the parameter~$\alpha_\ast$ be given as solution to the equation $\Theta(\alpha_\ast) = \delta/\sqrt 2$,  yields the convergence rate
$$\norm{\xd - x_{\alpha_\ast}^\delta} \leq 2 \varphi\lr{\Theta^{-1}(\delta/\sqrt 2)},
$$
which is known to be optimal in the ``low smoothness'' case, i.e., $\xd \in \range(A^*)$.

Under the same  condition on $\varphi$  we can also bound the right-hand side
in~(\ref{eq:axa-axp}) as
$$
\frac{ \norm{\ra(\asta)\lr{\asta}^{1/2}\xp}{}^{2}}{2\alpha}\leq
\frac{\lr{\sqrt{\alpha}\varphi(\alpha)}^{2}}{2\alpha} = \frac 1 2 \varphi^{2}(\alpha),
$$
which verifies \eqref{eq:residual-bound}.

We finally turn to discussing the maximal rate at which the
function~$\modi$ may tend to zero as~$\alpha\to 0$, provided
that~$\xd\neq 0$.
Considering the ratio~$\modi(\alpha)/\alpha$ we find
\begin{align*}
  \frac{\modi(\alpha)}\alpha  &\geq  \frac{
                   \norm{\ra(\asta)\lr{\asta}^{1/2}\xp}{}^{2}}{2\alpha^{2}} = \frac 1 {2\alpha^{2}}\int \frac{\alpha^{2}\lambda}{(\lambda +
  \alpha)^{2}} dE_\lambda \|\xd\|^2\\
& \geq \frac 1 8 \int_{\lambda \geq \alpha} \frac 1 \lambda dE_\lambda
  \|\xd\|^2= \frac 1 8 \norm{\chi_{[\alpha,\infty)}(\asta) \lr{\asta}^{-1/2}\xd}^{2}.
\end{align*}
This shows that either~$\xd \in \domain\lr{\asta}^{-1/2}$, and hence
that~$\xd \in\range\lr{A^{\ast}}$, in which case the right-hand side
is bounded away from zero (if~$\xd\neq 0$), or we have that~$\xd \not\in \range\lr{A^{\ast}}$,
and the right-hand side diverges. Hence, for nonzero $\xd$
the best attainable rate
near zero of the function~$\modi$ is linear as also predicted in Proposition~\ref{prop:dich}.

\subsection{ROF-Filter}
We consider the celebrated ROF-filter in image processing \cite{RuOsFA92}:
Let $\yd \in L^2(\R^2)$ represent an noisy image. Then a filtered version $x \in L^2(\R^2)\cap BV(\R^2)$
is computed by minimizing the Tikhonov functional
\[ \Tik(x;\yd) = \frac{1}{2}\norm{x- \yd}_{L^2(\R^2)}^2 + \alpha |x|_{TV}, \]
where $\reg(x) := |x|_{TV}$ denotes the total variation of $x$ on $\R^2$.
Obviously, this can be put into our framework with $A$ being the embedding
operator from $BV(\R^2)$ to $L^2(\R^2)$.

For some special cases, where $\xd$ is the characteristic function of simple
geometric shapes, the minimizers can be computed explicitly.
Denote by $B_{x,R}$ a ball with center $x$ and radius $R$.
Consider first the case when $\xd$ is the characteristic function of a
ball $B_{0,R}$:
\[ \xd(s)   = \chi_{B_{0,R}}(s) := \begin{cases} 1 & \text{if } \|s\|_{\R^2} \leq 1, \\
0 & \text{else}. \end{cases} \]
The minimizer $\xa$ of $\Tik(.;y)$ with exact data is given by, e.g., \cite{Me01}
\[ \xa = \max\{1- \tfrac{2 \alpha}{R},0\}  \chi_{B_{0,R}}(s).  \]
Calculating the index function in Assumption~\ref{ass:conv} is now a simple task
as $|\chi_{B_{0,R}}|_{TV} = 2 \pi R$
\begin{align*} \reg(\xd) - \reg(\xa) &=  \modi(\alpha) = 2 \pi R\left(1 - \max\{1- \tfrac{2 \alpha}{R},0\} \right) \\
&=2 \pi R \min\{\frac{2 \alpha}{R},1\}
 = 4 \pi \alpha \quad \mbox{ if } \alpha < R/2.
\end{align*}
For a comparison, we may compute the Bregman distances. For the asymptotically interesting
case, $\alpha < \frac{R}{2}$ we find that
\[ \Bj{\xi_\alpha}(\xa;\xd) = \Bj{\xi_\alpha}(\xd;\xa) = 0 \qquad \forall \alpha < \frac{R}{2}, \]
which yields a trivial rate, but of course, does not violate the
upper bound $\mod(\alpha)$ in \eqref{eq:upperbound2} for $\delta= 0$.
The squared norm of the residual  for $\alpha < \frac{R}{2}$
is given by
\[ \norm{A \xa - y}^2 =   4 \pi  \alpha^2, \]
hence, \eqref{eq:residual-bound} clearly holds.
We also observe that a variational inequality of the form~\eqref{eq:vi-Phi}, or~\eqref{eq:kindermann} below,
holds with $\mod(s) \sim s$.

For noisy data, $\xad$ cannot be calculated analytically,
but our results
suggest for such $\xd$ a suitable parameter choice of the form $\alpha = \delta$, which provides a convergence
 rate
\[ \Bj{\xi_\alpha^\delta}(\xad;\xd) \leq (4 \pi +\frac{1}{2} ) \delta. \]

A less simple situation appears when the exact solution is
the characteristic function of the unit square
\[ \xd(s) = \chi_{[0,1]^2}(s) = \begin{cases}1 & \text{ if } s \in [0,1]^2 \\
0 &\text{else} \end{cases}. \]
An explicit solution is known here as well \cite{Chetal10}.
For $R>0$ define the rounded square
\[ C_R := \bigcup_{x: B_{x,R} \subset [0,1]^2} B_{x,R},  \]
which has the shape of a square with the four
corners cut off and replaced by circular arcs of radius $R$ that
meet tangentially the edges of the square.
The solution satisfies $0\leq \xa \leq 1$ and
can be characterized by the level sets: for $s \in [0,1]$
\begin{align*}
    \{\xa > s\} = \begin{cases} \emptyset & \text{ if }
    s \geq 1-\frac{\alpha}{R^*} \\
    C_{\frac{\alpha}{1-s}} & \text{ if } s
    \leq 1-\frac{\alpha}{R^*} \end{cases}.
\end{align*}
Here $R^*$ is a limiting value, which can be computed explicitly. Since we are interested in the
asymptotics $\alpha \to 0$, we generally impose
the condition $\alpha \leq R^*$ as otherwise
$\xa = 0$.
The index function $\modi$ can now be calculated
by the coarea formula
\begin{align*}
\reg(\xd) - \reg(\xa) = \modi(\alpha) =
4 - \int_0^{1-\frac{\alpha}{R^*} }
|C_{\frac{\alpha}{1-s}}|_{TV} ds
\end{align*}
The value of $|C_{R}|_{TV}$ is its perimeter
and can be calculated by elementary geometry
to $|C_{R}|_{TV} = 4 - 2(4-\pi)R$. Thus, evaluating the integral,  we obtain
\begin{align*}
\modi(\alpha)   = \frac{4}{R^*} \alpha +
2(4-\pi) \alpha\lr{\log\lr{\frac{R^*}{\alpha}}} \qquad \alpha \leq R^*.
\end{align*}
Thus, in this case,
\[ \modi(\alpha)\sim \alpha \log(1/\alpha)
\qquad \text{ as } \alpha \to 0. \]
The residual norm is given by
\begin{align*}
    \|A \xa - y\|^2 &=
    \|\xa -\xd\|_{L^2}^2 \\
    &=
    \frac{\alpha^2}{{R^*}^2} +  2(4 -  \pi) \alpha^2\left(\log\lr{\frac{R^*}{\alpha}}\right)
    \qquad \alpha < R^*.
\end{align*}
Obviously, the bound \eqref{eq:residual-bound} is satisfied.
The approximation error in the Bregman distance (with our choice of the subgradient element)
is hence given by
\begin{align*} \Bj{\xi_\alpha}(\xa;\xp)
&= \reg(\xd) - \reg(\xad) -  \frac{1}{\alpha} \|A \xa - y\|^2
 = \frac{3 }{{R^*}^2} \alpha   \qquad \alpha < R^*.
\end{align*}

We observe that, for the square,
the parameter choice that minimizes the upper bound
\eqref{eq:upperbound2} differs from that for the ball
as we have that $\alpha \sim C \frac{\delta}{(\log(1/\delta))^\frac{1}{2}}$,
which highlights the (well-known) dependence of the parameter choice on the
regularity of the exact solution.

Note also, that the decay of the
Bregman distance $\Bj{\xi_\alpha}(\xa;\xp)$ alone
does not suit well as a measure of regularity for $\xd$ since the
logarithmic factor that appears in the condition in
Assumption~\ref{ass:conv}  is not observed for this Bregman distance.

\subsection{On $\ell^1$-regularization when sparsity is slightly missing}

We consider the \emph{injective} continuous linear operator $A\colon \ell^1\to \ell^2$ and the penalty~$J(x):=\|x\|_{\ell^1} = \norm{x}_1$. Notice that $\ell^1=c_0^*$, it thus has a predual, and we assume that  $A$ is weak$^*$-to-weak continuous, and  the penalty~$J$ is stabilizing in this sense (see also~\cite{FleGer17}).

The crucial additional assumption on the operator~$A$  is that the unit elements $e^{(k)}$ with $e_k^{(k)}=1$  and $e_i^{(k)}=0$ for $i \not=k$, satisfy  source conditions $e^{(k)}=A^*f^{(k)},\;f^{(k)} \in Y$ for all $k \in \mathbb{N}$ .
Under these assumptions, and with~$\xd =(x_k^\dagger)_{k \in \mathbb{N}}\in X$ from (\ref{eq:opeq}), we assign the function
\begin{equation} \label{eq:BFHPhi}
\mod(t)=2 \inf \limits_{n \in \mathbb{N}} \left(\sum \limits_{k=n+1}^\infty |x_k^\dagger| + t\,\sum \limits_{k=1}^n \|f^{(k)}\|_Y  \right),\quad t>0.
\end{equation}
Notice that the function~$\mod$ from~\eqref{eq:BFHPhi} is a concave index function. It was shown in~\cite{BurFleHof13} that then a variational
inequality of the form
\begin{equation*}
\|x-\xd\|_X \le  \|x\|_X-\|\xd\|_X + \mod\lr{\norm{A\xd - Ax}} \quad \mbox{for all} \quad  x\in X
\end{equation*}
holds true. This immediately implies the validity of the condition (\ref{eq:vi-Phi}) with the same index function $\mod$, and an application of~item (3) of Proposition~\ref{pro:peter} shows that the error estimate~(\ref{eq:upperbound2}) is valid for that $\mod$.

The behavior of the index function $\mod$ from (\ref{eq:BFHPhi}) essentially depends on the decay rate of the tail of $x_k^\dagger \to 0$ of the solution element~$\xd$. When sparsity is (slightly) missing, then the function~$\mod$ will be strictly concave. However, if~$\xd$ is sparse, i.e.,\ $x_k^\dagger=0 \quad \mbox{for} \quad k>n_{max}$, then the function~$\mod$ reduces to the linear function
$$
\mod(t)= \left(\sum \limits_{k=1}^{n_{max}} \|f^{(k)}\|_Y \right)\,t, \quad t >0.
$$
As Example~\ref{xmpl:monomial} highlights, this results in a linear companion  function~$\modi$. Thus Theorem~\ref{thm:mod} applies, and the choice of~$\alpha \sim \delta$ yields a rate for the Bregman distance $\Bj{\xi_\alpha^\delta}(\xad;\xd) = \mathcal O(\delta)$ as~$\delta\to 0$ in the sparse case.

\section{Outlook to higher order rates}
\label{sec:outlook}
There might be a way for overcoming the limitation of sub-linear functions~$\modi$ in the assumptions~\ref{ass:conv} or~\ref{ass:2prime}. The underlying observation for this is the identity
\begin{equation} \label{eq:noiseminus}
B_{\xi_{\alpha}}\lr{x_{\alpha};\xd} = \frac{2}{\alpha}(T_\alpha(\xd,y)-T_\alpha(\xa,y))-(J(\xd)-J(\xa).
\end{equation}
The right-hand side above is again entirely based on noise-free quantities, and its decay could be used as smoothness assumption.

If one could prove that there were an inequality of the form
\begin{equation} \label{eq:noiseplus}
B_{\xi_{\alpha}^\delta}\lr{x_{\alpha}^\delta;\xd} \le C_1\,B_{\xi_{\alpha}}\lr{x_{\alpha};\xd}+ C_2\,\delta^2/\alpha,\quad \alpha>0,
\end{equation}
with positive constants $C_1$ and $C_2$, then this might open the pathway for higher order rates. Indeed, in Hilbert space~$X$ and for the standard penalty~$\reg(x) := \tfrac 1 2 \norm{x}_X^2$, cf. Section~\ref{subsec:quadratic}, we find that  $B_{\xi_{\alpha}^\delta}\lr{x_{\alpha}^\delta;\xd}=\|\xad-\xd\|_X^2$, and hence that the inequality~(\ref{eq:noiseplus})
is satisfied with $C_1=2$ and $C_2=1$. Moreover,
one can easily verify that
$$
\frac{2}{\alpha}(T_\alpha(\xd,y)-T_\alpha(\xa,y))-(J(\xd)-J(\xa) = \frac 1 2 \norm{\ra(\asta)\xd}^2,
$$
with~$\ra(\asta)= \alpha\lr{\alpha + \asta}^{-1}$, being the (squared) residual for (standard linear) Tikhonov regularization. This squared residual is known to decay  of order up to~$\mathcal O(\alpha^2)$ as $\alpha\to 0$, which then allows for higher rates
$B_{\xi_{\alpha}^\delta}\lr{x_{\alpha}^\delta;\xd}=\mathcal{O}(\delta^{4/3})$, attained under the limiting source condition $\xd=A^*Aw,\;w \in X$, and for the a priori parameter choice $\alpha \sim \delta^{2/3}$.
It is thus interesting to see whether and under which additional assumptions an inequality of the form~\eqref{eq:noiseplus} holds.

\appendix
\section{Proofs}
\label{sec:proofs}
Let us  define the noisy and exact residuals, and the noise term as
\begin{equation}\label{nr}
\rad:= A\xad -\yd, \ \ra := A\xa - y = A (\xa - \xd)\quad
\text{and}\ \Delta:= \yd - y.
\end{equation}
Notice that all quantities~$\ra,\rad$ as well as~$\Delta$ belong to
the Hilbert space~$H$.
The subsequent analysis will be based on the optimality conditions (recall our convention on the choice of
$\xi_\alpha^\delta \in \partial J(\xad)$ and
$\xi_\alpha \in \partial J(\xa)$)
\begin{align}
\langle A \xad - \yd,A w\rangle + \alpha \langle \xi_\alpha^\delta, w \rangle
  & = 0,  \qquad \forall w \in X,  \label{opc1} \\
\langle A \xa - A \xd,A w\rangle + \alpha \langle \xi_\alpha, w \rangle
  & = 0,  \qquad \forall w \in X.  \label{opc2}
\end{align}
In particular, the optimality conditions lead to the following
formulas, by
\begin{enumerate}
\item subtracting  \eqref{opc1} from \eqref{opc2} using $w = \xa
  -\xd$,
\item  using \eqref{opc1} with $w = \xa -\xad $, and
\item using \eqref{opc2} with $w = \xa -\xd $,
\end{enumerate}
respectively:
\begin{align}
\langle  \xi_\alpha- \xi_\alpha^\delta,\xd-\xa \rangle &=  - \frac 1 \alpha\langle \rad  - \ra, \ra \rangle  \label{est2} \\
-\langle \xi_\alpha^\delta,\xa -\xad  \rangle  &= \frac{1}{\alpha}
\langle \rad, \ra -\rad  - \Delta) \rangle  \label{est2b}  \\
 -   \langle \xi_\alpha, \xd -\xa \rangle  &= -\frac{1}{\alpha} \|\ra\|^2.   \label{est2c}
\end{align}
The following bounds will be the key for proving Theorem~\ref{thm:mod}.
\begin{lem}\label{lem:distance-bounds}
  Under Assumption~\ref{ass:2prime} we have
  \begin{enumerate}
  \item\label{it:bxaxp} $ \Bj{\xi_\alpha}(\xa;\xp) \leq \modi(\alpha) - \frac 1
    {2\alpha}\norm{\ra}^{2}$.
  \item $\Bj{\xi_\alpha^\delta}(\xad;\xa) \leq \frac{\delta^{2}}{2\alpha} - \frac 1
    {2\alpha}\norm{\ra}^{2} + \frac 1 \alpha \scalar{\rad}{\ra} $
  \end{enumerate}
\end{lem}
\begin{proof}
  Using the optimality condition \eqref{est2c} we find that
\begin{equation*}\label{ali:after11}
  \begin{split}
 &\Bj{\xi_\alpha}(\xa;\xd) +   \frac{1}{2\alpha} \|\ra\|^2 \\
  & \qquad = \reg(\xd) - \reg(\xa) - \langle \xi_\alpha,\xd-\xa \rangle
  + \frac{1}{2\alpha} \|\ra\|^2 \\
  & \qquad = \reg(\xd) - \reg(\xa) -   \frac{1}{2\alpha} \|\ra\|^2 \\
  & \qquad = \frac{1}{\alpha} \left(\Tik(\xd;y) - \Tik(\xa;y) \right)
  \leq \modi(\alpha),
  \end{split}
  \end{equation*}
which proves the first assertion.

For proving the second assertion we use the definition
of~$\Bj{\xi_\alpha^\delta}(\xad;\xa)$ and~(\ref{est2b}) to find
\begin{equation}
  \label{eq:bxadxd1}
\Bj{\xi_\alpha^\delta}(\xad;\xa) 
=  \reg(\xa) - \reg(\xad) + \frac 1 \alpha \langle \rad, \ra -\rad  - \Delta) \rangle.
\end{equation}
The minimizing property of~$\xa$ also yields
\begin{align*}
\reg(\xa) - \reg(\xad) &\leq \frac 1{2 \alpha} \left[ \norm{A \xad - y}^{2} -
  \norm{A \xa - y}^{2}\right] \\
&= \frac 1{2 \alpha}\left[ \norm{\rad + \Delta}^{2} - \norm{\ra}^{2}\right]
\end{align*}
We rewrite
$$
\langle \rad, \ra -\rad  - \Delta) \rangle = - \norm{\rad}^{2} +
\scalar{\rad}{\ra - \Delta}.
$$
Using this and plugging the above estimate into~(\ref{eq:bxadxd1}) gives
\begin{align*}
  \Bj{\xi_\alpha^\delta}(\xad;\xa) &\leq \frac 1 {2\alpha} \left[ \norm{\rad + \Delta}^{2}
  - \norm{\ra}^{2}\right] - \frac 1 \alpha  \norm{\rad}^{2}
+ \frac 1 \alpha \scalar{\rad}{\ra - \Delta}\\
& =  \frac 1 {2\alpha} \norm{\Delta}^{2}
-\frac 1 {2\alpha} \norm{\ra}^{2}-\frac 1 {2\alpha} \norm{\rad}^{2} +
  \frac 1 \alpha \scalar{\rad}{\ra}\\
& \leq  \frac {\delta^{2}} {2\alpha} -\frac 1 {2\alpha} \norm{\ra}^{2}
+\frac 1 \alpha \scalar{\rad}{\ra},
\end{align*}
completing the proof of the second assertion and of the lemma.
\end{proof}
We are now in a position to give detailed proofs of the results in
Section~\ref{sec:main}.

\begin{proof}[Proof of Proposition~\ref{thm:equivalence}]
From item \eqref{it:bxaxp} of Lemma~\ref{lem:distance-bounds} we know
that, for all $\alpha>0$,
$$
\frac{\norm{A \xa - y}{}^{2}}{2\alpha} =
\frac{\norm{\ra}^{2}}{2\alpha}\leq \modi(\alpha).
$$
Therefore, Assumption~\ref{ass:2prime} implies that
\begin{align*}
  \reg(\xp) - \reg(\xa) &= \reg(\xp) - \reg(\xa) - \frac{\norm{A \xa -
                    y}{}^{2}}{2\alpha} +
                    \frac{\norm{\ra}^{2}}{2\alpha}\\
&\leq \modi(\alpha) + \frac{\norm{\ra}^{2}}{2\alpha} \leq 2 \modi(\alpha),
\end{align*}
which completes the proof.
\end{proof}

\begin{proof}[Proof of Proposition~\ref{prop:dich}]
First, if~$\inf_{x\in X}\reg(x) = \reg(\xd)$ then
$$
\alpha J(\xd)=\Tik(\xd;A\xd) \ge \Tik(\xa;A\xd)\geq \alpha \reg(\xa) \geq \alpha \reg(\xd),
$$
and for all $\alpha>0$ we have $J(\xa)= J(\xd)$.  This allows us to prove the assertion in the first (singular) case.
Otherwise, assume to the contrary that there is an index function~$\modi$, and for a decreasing sequence of regularization parameters $(\alpha_k)_k$ with
$\lim_{k\to\infty} \alpha_k = 0$
the limit condition
\[ \lim_{k\to \infty} \frac{\modi(\alpha_k)}{\alpha_k} = 0 \]
holds.  Consequently, we find from~\eqref{eq:residual-bound} that
$\lim_{k\to \infty} \frac{1}{\alpha_k}\norm{A x_{\alpha_k}-y} = 0$, with~$y=A \xd$.

Due to the optimality condition~\eqref{opc2} for $\xa$ we have that
\[  A^*(Ax_{\alpha_k}-y)+\alpha_k \xi_{\alpha_k}=0 \quad \mbox{for some} \quad \xi_{\alpha_k} \in \partial \reg(x_{\alpha_k}) \subset X^*. \]
This yields
\[ \xi_{\alpha_k} = -\frac1{\alpha_k} A^*(Ax_{\alpha_k}-y). \]
Since $A^*:H \to X^*$ is a bounded linear operator, we get
\[ \|\xi_{\alpha_k}\|_{X^*} \le \frac1{\alpha_k}\|A^*\|_{\mathcal{L}(H,X^*)}\,\|Ax_{\alpha_k}-y\|\to 0\quad
\text{ as } \quad k\to\infty. \]
Since $\xi_{\alpha_k}\in \partial \reg(x_{\alpha_k})$,
and after taking the limit, we find for all $y \in H$ that
\[ \reg(y) \ge \limsup_{k\to\infty}\left\{ \reg(x_{\alpha_k})+\left\langle \xi_{\alpha_k},y-x_{\alpha_k}\right\rangle\right\} = \limsup_{k\to\infty} \reg(x_{\alpha_k}) = \reg(\xd), \]
where we used  Assumption~\ref{ass:conv} and $\lim_{k\to\infty} \modi(\alpha_k)= 0$. Thus, we conclude that~$\inf_{x\in X}\reg(x) = \reg(\xd)$. This contradicts the assumption, and hence the function~$\modi$ cannot decrease to zero super-linearly as $\alpha\to0$.
\end{proof}

\begin{proof}
  [Proof of Theorem~\ref{thm:mod}]
Here we recall the three-point identity (see, e.g.,  \cite{ScKaHoKa12}). For $u,v,w\in X$ and $\xi \in \partial \reg(w)$, $\eta \in \partial \reg(v)$,
we have that       
\begin{align}
\Bj{\xi}(w;u) =  \Bj{\eta}(v;u) + \Bj{\xi}(w;v) +
\langle \eta - \xi,u -v \rangle, \label{bregone}
\end{align}
and this specifies with~$u:= \xd,\ v:= \xa$, and~$w:= \xad$ to
\begin{align}
\Bj{\xi_\alpha^\delta}(\xad;\xd) &=  \Bj{\xi_\alpha}(\xa;\xd) + \Bj{\xi_\alpha^\delta}(\xad;\xa) + \langle \xi_\alpha - \xi_\alpha^\delta,\xd -\xa \rangle, \label{est1}
\end{align}
Inserting \eqref{est2} into \eqref{est1} gives
\begin{equation*} 
\begin{split}
\Bj{\xi_\alpha^\delta}(\xad;\xd) &= \Bj{\xi_\alpha} (\xa;\xd) + \Bj{\xi_\alpha^\delta}(\xad;\xa) - \frac{1}{\alpha}  \langle \rad  - \ra, \ra \rangle  \\
\end{split}
\end{equation*}
An application of the bounds in Lemma~\ref{lem:distance-bounds}
provides us with the estimate
\begin{align*}
 \Bj{\xi_\alpha^\delta}(\xad;\xd) &\leq \modi(\alpha) + \frac {\delta^{2}} {2\alpha}
-\frac 1 {\alpha} \norm{\ra}^{2} +\frac 1 \alpha \scalar{\rad}{\ra}-
  \frac{1}{\alpha}  \langle \rad  - \ra, \ra \rangle\\
&=  \modi(\alpha) + \frac {\delta^{2}} {2\alpha} ,
\end{align*}
and the proof is complete.
\end{proof}

\section{Some convex analysis for index functions}
\label{sec:convex-analysis}

We shall provide some additional details for convex index functions.
First, it is well known that for convex index function~$f$ we have
that~$0 < s\leq t$ yields~$f(s)/s\leq f(t)/t$. Indeed, we
let~$0 < \theta:= s/t \leq 1$ and obtain that
$$
f(s) = f(\theta t + (1-\theta) 0)\leq \theta f(t) + (1-\theta) f(0) =
\frac s t f(t),
$$
which allows us to prove the assertion. This implies that the limit~$g:= \lim_{t\to
  0}f(t)/t\geq 0$ exists. If $g>0$ then $f$ is linear near zero, and
this case is not interesting in this study. Otherwise, we assume
that~$g=0$. In this interesting (sub-linear) case the following result
is relevant.
\begin{lem}\label{lem:convex}
  Suppose that~$f$ is a convex index function.
The following assertions are equivalent.
  \begin{enumerate}
  \item\label{it:direct} The quotient~$f(t)/t,\ t>0$ is a strictly
    increasing  index function.
  \item\label{it:inverse} There is a strictly increasing index function~$\varphi$, and
    the companion~$\Theta(t) := \sqrt t \varphi(t),\ t>0$  such that
    the representation~$f(t) =
    \Theta^{2}\lr{\lr {\varphi^{2}}^{-1}(t)},\ t>0$ is valid.
  \end{enumerate}
\end{lem}
\begin{proof}
  Clearly, if~$f$ has a representation as in~(\ref{it:inverse}) then
  we find, with letting~$\varphi^{2}(s) = t$, that
$$
\frac{f(t)} t = \frac{\Theta^{2}(s)}{\varphi^{2}(s)} =  s\searrow 0,
$$
as~$s\to 0$.

For the other implication we observe that by assumption we can
(implicitly) define the strictly increasing index function~$\varphi$ by
\begin{equation}
  \label{eq:varphi-def}
  \varphi\lr{\frac{f(t)}t} \colon =  \sqrt t,\quad t>0.
\end{equation}
This yields that
$$
f(t) = t \lr{\varphi^{2}}^{-1}(t) =
\Theta^{2}\lr{{\lr{\varphi^{2}}^{-1}(t)}},\quad t>0,
$$
which completes the proof.
\end{proof}

As an interesting consequence we mention the following result for the Fenchel conjugate function~$f^\ast$ to the convex (index) function~$f$, which is defined as
\begin{equation}
    \label{eq:Fconjugate}
    f^\ast(t) := \sup_{s\geq 0}\lr{s t - f(s)}, \quad t>0.
\end{equation}
Clearly, both functions~$f$ and its conjugate~$f^\ast$ obey the \emph{Fenchel-Young Inequality}
\begin{equation}
    \label{eq:FYI}
    s t  \leq f(s) + f^\ast(t),\quad s,t\geq 0.
\end{equation}
\begin{cor}\label{cor:appendix}
  Suppose that~$f$ is a convex index function such that the
  quotient~$f(t)/t,\ t>0$ is a strictly increasing  index function.
  Then the Fenchel conjugate function~$f^{\ast}$ is an index function,
  and there is a strictly increasing index function~$\varphi$ such
  that
$$
\frac{f^{\ast}(t)}{t} \leq \varphi^{2}(t),\quad t>0.
$$
\end{cor}
\begin{proof}
  First, by Lemma~\ref{lem:convex} there is a strictly increasing
  index function~$\varphi$ such that~$f(t) =
    \Theta^{2}\lr{\lr {\varphi^{2}}^{-1}(t)},\ t>0$.
Now we use the ``poor man's Young Inequality''  of the form
$$
\varphi^{2}(x) y \leq \varphi^{2}(x) x + \varphi^{2}(y) y,\quad x,y>0,
$$
which in turn, by letting~$s:= \varphi^{2}(x)$ and~$t:= y$, implies
$$
s t \leq \Theta^{2}\lr{\lr{\varphi^{2}}^{-1}(s)} + \Theta^{2}(t),\quad s,t>0.
$$
For the Fenchel conjugate~$f^{\ast}$ this yields
$$
f^{\ast}(t) := \sup_{s>0}\set{st - f(s)} \leq \Theta^{2}(t).
$$
From this bound we conclude that~$f^{\ast}$ will be an index function
for which the quotient~$f^{\ast}(t)/t$ has the desired bound.
\end{proof}
\section*{Acknowledgment}
We thank Peter Elbau (University of Vienna) and Jens Flemming (TU Chemnitz) for suggesting to us essential ingredients for
the proofs of Propositions~\ref{prop:dich} and~\ref{prop:nosuper}, respectively.\\
The research of the first author was supported by Deutsche
Forschungsgemeinschaft (DFG-grant HO 1454/12-1).
The research of the second author was supported
by the Austrian Science Fund (FWF)
project P~30157-N31.


\begin{thebibliography}{10}

\bibitem{AlbElb16}
{\sc V.~Albani, P.~Elbau, M.~V. de~Hoop, and O.~Scherzer}, {\em Optimal
  convergence rates results for linear inverse problems in {H}ilbert spaces},
  Numer.~Funct.~Anal.~Optim., 37 (2016), pp.~521--540.

\bibitem{AndElb15}
{\sc R.~Andreev, P.~Elbau, M.~de~Hoop, L.~Qiu, and O.~Scherzer}, {\em
  Generalized convergence rates results for linear inverse problems in
  {H}ilbert spaces}, Numer. Funct. Anal. Optim., 36 (2015), pp.~549--566.

\bibitem{AnzHofMat14}
{\sc S.~W. Anzengruber, B.~Hofmann, and P.~Math{\'e}}, {\em Regularization
  properties of the sequential discrepancy principle for {T}ikhonov
  regularization in {B}anach spaces}, Appl. Anal., 93 (2014), pp.~1382--1400.

\bibitem{BurFleHof13}
{\sc M.~Burger, J.~Flemming, and B.~Hofmann}, {\em Convergence rates in
  {$\ell\sp 1$}-regularization if the sparsity assumption fails}, Inverse
  Problems, 29 (2013), p.~025013 (16pp).

\bibitem{BurOsh04}
{\sc M.~Burger and S.~Osher}, {\em Convergence rates of convex variational
  regularization}, Inverse Problems, 20 (2004), pp.~1411--1421.

\bibitem{Chetal10}
{\sc A.~Chambolle, V.~Caselles, D.~Cremers, M.~Novaga, and T.~Pock}, {\em An
  introduction to total variation for image analysis}, in Theoretical
  foundations and numerical methods for sparse recovery, vol.~9 of Radon Ser.
  Comput. Appl. Math., Walter de Gruyter, Berlin, 2010, pp.~263--340.

\bibitem{EHN96}
{\sc H.~W. Engl, M.~Hanke, and A.~Neubauer}, {\em Regularization of inverse
  problems}, vol.~375 of Mathematics and its Applications, Kluwer Academic
  Publishers Group, Dordrecht, 1996.

\bibitem{Fl12}
{\sc J.~Flemming}, {\em {G}eneralized {T}ikhonov {R}egularization and {M}odern
  {C}onvergence {R}ate {T}heory in {B}anach {S}paces}, Shaker Verlag, Aachen,
  2012.

\bibitem{Fl17}
\leavevmode\vrule height 2pt depth -1.6pt width 23pt, {\em A converse result
  for {B}anach space convergence rates in {T}ikhonov-type convex regularization
  of ill-posed linear equations}, J. Inverse Ill-Posed Probl., 26 (2018).
\newblock https://doi.org/10.1515/jiip-2017-0116; aop.

\bibitem{Fl18}
{\sc J.~Flemming}, {\em Existence of variational source conditions for
  nonlinear inverse problems in {B}anach spaces}, J. Inverse Ill-Posed Probl.,
  26 (2018), pp.~277--286.

\bibitem{FleGer17}
{\sc J.~Flemming and D.~Gerth}, {\em Injectivity and weak$^*$-to-weak
  continuity suffice for convergence rates in $\ell^1$-regularization},
  J.~Inv.~Ill-Posed Probl., 26 (2018), pp.~85--94.

\bibitem{FleHof10}
{\sc J.~Flemming and B.~Hofmann}, {\em A new approach to source conditions in
  regularization with general residual term}, Numer.~Funct.~Anal.~Optim., 31
  (2010), pp.~254--284.

\bibitem{FHM11}
{\sc J.~Flemming, B.~Hofmann, and P.~Math\'e}, {\em Sharp converse results for
  the regularization error using distance functions}, Inverse Problems, 27
  (2011), pp.~025006, 18.

\bibitem{Gr10}
{\sc M.~Grasmair}, {\em Generalized {B}regman distances and convergence rates
  for non-convex regularization methods}, Inverse Problems, 26 (2010).

\bibitem{HoKaPoSc07}
{\sc B.~Hofmann, B.~Kaltenbacher, C.~P{\"o}schl, and O.~Scherzer}, {\em A
  convergence rates result for {T}ikhonov regularization in {B}anach spaces
  with non-smooth operators}, Inverse Problems, 23 (2007), pp.~987--1010.

\bibitem{HofMat07}
{\sc B.~Hofmann and P.~Math\'e}, {\em Analysis of profile functions for general
  linear regularization methods}, SIAM J.~Numer.~Anal., 45 (2007),
  pp.~1122--1141.

\bibitem{HoMa12}
\leavevmode\vrule height 2pt depth -1.6pt width 23pt, {\em Parameter choice in
  {B}anach space regularization under variational inequalities}, Inverse
  Problems, 28 (2012), pp.~104006, 17.

\bibitem{ItoJin15}
{\sc K.~Ito and B.~Jin}, {\em Inverse Problems: Tikhonov Theory and
  Algorithms}, vol.~22 of Series on Applied Mathematics, World Scientific
  Publishing Co. Pte. Ltd., Hackensack, NJ, 2015.

\bibitem{Ki16}
{\sc S.~Kindermann}, {\em Convex {T}ikhonov regularization in {B}anach spaces:
  new results on convergence rates}, J. Inverse Ill-Posed Probl., 24 (2016),
  pp.~341--350.

\bibitem{MaPe03}
{\sc P.~Math\'e and S.~V. Pereverzev}, {\em Geometry of linear ill-posed
  problems in variable {H}ilbert scales}, Inverse Problems, 19 (2003),
  pp.~789--803.

\bibitem{Me01}
{\sc Y.~Meyer}, {\em {Oscillating Patterns in Image Processing and Nonlinear
  Evolution Equations}}, {AMS, Providence, RI}, 2001.

\bibitem{NHHKT10}
{\sc A.~Neubauer, T.~Hein, B.~Hofmann, S.~Kindermann, and U.~Tautenhahn}, {\em
  Improved and extended results for enhanced convergence rates of {T}ikhonov
  regularization in {B}anach spaces}, Appl.~Anal., 89 (2010), pp.~1729--1743.

\bibitem{Resm05}
{\sc E.~Resmerita}, {\em Regularization of ill-posed problems in {B}anach
  spaces: convergence rates}, Inverse Problems, 21 (2005), pp.~1303--1314.

\bibitem{ResSch06}
{\sc E.~Resmerita and O.~Scherzer}, {\em Error estimates for non-quadratic
  regularization and the relation to enhancement}, Inverse Problems, 22 (2006),
  pp.~801--814.

\bibitem{RuOsFA92}
{\sc L.~I. Rudin, S.~Osher, and E.~Fatemi}, {\em Nonlinear total variation
  based noise removal algorithms}, Physica D, 60 (1992), pp.~259--268.

\bibitem{Scherzer09}
{\sc O.~Scherzer, M.~Grasmair, H.~Grossauer, M.~Haltmeier, and F.~Lenzen}, {\em
  Variational Methods in Imaging}, vol.~167 of Applied Mathematical Sciences,
  Springer, New York, 2009.

\bibitem{ScKaHoKa12}
{\sc T.~Schuster, B.~Kaltenbacher, B.~Hofmann, and K.~S. Kazimierski}, {\em
  Regularization methods in {B}anach spaces}, vol.~10 of Radon Series on
  Computational and Applied Mathematics, Walter de Gruyter GmbH \& Co. KG,
  Berlin, 2012.

\end{thebibliography}
\end{document}